\documentclass{amsart}

\usepackage{amsmath}
\usepackage{amsfonts}
\usepackage{amssymb}
\usepackage{amsthm}
\usepackage{hyperref}
\newtheorem*{WDF}{Weyl Dimension Formula}

\newtheorem*{thm}{Theorem}

\newtheorem*{prop}{Proposition}
\newtheorem*{lem}{Lemma}

\newtheorem*{cor}{Corollary}
\numberwithin{equation}{section}

\title{An explicit formula for the Hilbert series of a partial flag variety}
\author{Wayne A. Johnson}
\address{{\bf Wayne A. Johnson}\\
Department of Mathematics\\
Truman State University\\
100 E Normal Ave\\
Kirksville, MO 63501\\
email: wjohnson@truman.edu}

\begin{document}

\begin{abstract}
For a semisimple, simply-connected linear algebraic group, $G$, and parabolic subgroup, $P\subseteq G$, we use the fact that the Hilbert polynomial of the equivariant embedding of $G/P$ is equal to the Hilbert function to compute an explicit formula for the Hilbert series of $G/P$ in terms of the dimensions of finitely many irreducible representations of $G$. As an example, we compute the Hilbert series of the adjoint variety of $SL(n+1,\mathbb{C})$. We conclude by computing the linear term of the numerator of the Hilbert series of any fundamental representation in type $A$.

\end{abstract}

\maketitle

\section{Introduction}

Throughout this paper, $G$ will be a semi-simple, simply-connected linear algebraic group over $\mathbb{C}$, and $P\subseteq G$ will be a parabolic subgroup of $G$. The partial flag variety $G/P$ (and its embedding in projective space in particular) have been extensively studied. In \cite{GW}, Gross and Wallach compute the Hilbert polynomial and Hilbert series of the $G$-equivariant embedding of $G/P$ into projective space. The Hilbert polynomial is shown to be given by the Weyl Dimension Formula, and the Hilbert series is show to be a finite product of simple differential operators acting on a geometric series (the relevant details will be recalled later in this section). This formula is not explicit in the sense that it does not immediately give the Hilbert series. Instead it gives a process for obtaining the rational form of the series.

In this paper, we use the result on the Hilbert polynomial from \cite{GW} to explicitly compute the Hilbert series of $G/P$ in terms of the dimensions of $d$ nontrivial, irreducible representations of $G$, where $d$ is the dimension of $G/P$. Especially in low dimension, this formula can be used to explicitly compute the Hilbert series of $G/P$. In particular, we aim to prove the following theorem. The notation used will be defined in in the next subsection.

\begin{thm}
Let $d=\dim(G/P_\lambda)$. Assume $g(x)=a_0+a_1x+\dots+a_dx^d$ is the numerator polynomial in
\begin{center}
$HS_\lambda(x)=\displaystyle\frac{g(x)}{(1-x)^{d+1}}$.
\end{center}
Then $a_0=1$, and for $1\leq i\leq d$, we have
\begin{center}
$a_i=\displaystyle\sum_{j=0}^i(-1)^j{d+1\choose j}D_{i-j}$,
\end{center}
where $D_j:=\dim(L(j\lambda))$ for $j=0,\dots,d$.
\end{thm}

Equivariant embeddings of $G/P$ include some of the most important projective varieties, including the Veronese embeddings, the Segre embeddings of products of projective spaces, and the Pl\"ucker embeddings of Grassmannians (see \cite{GW}, \cite{JH}).

The rest of the paper is structured as follows. In the remainder of this first section, we review the necessary background for the proof of the main theorem, including an overview of the results in \cite{GW}. \S2 contains the proof of the main theorem. \S3 contains an in-depth example showing how the main theorem can be used to compute the Hilbert series of the adjoint variety of $SL(n+1,\mathbb{C})$. The concluding section includes a computation of $a_1$ for any fundamental representation of $SL(n+1,\mathbb{C})$ in terms of the difference between binomial coefficients and the entries of the so-called rascal triangle.

\subsection{Preliminary results on partial flag varieties}

Fix a Borel subgroup, $B\subset G$, and a maximal torus, $T\subset B$. Denote by $\mathfrak{g}$, $\mathfrak{h}$, and $\mathfrak{b}$ the Lie algebras of $G$, $T$, and $B$, respectively. Let $\Phi$ be the roots given by the pair $(\mathfrak{g},\mathfrak{h})$, and denote by $\Phi^+$ the set of positive roots corresponding to $\mathfrak{b}$.

Let $P_+(G)$ denote the set of dominant integral weights of $G$. For any $\lambda\in P_+(G)$, let $L(\lambda)$ denote the irreducible representation of $G$ with highest weight $\lambda$. In particular, if $j\in\mathbb{N}$, we let $L(j\lambda)$ denote the irreducible representation of $G$ with highest weight $j\lambda$. The dimension of $L(\lambda)$ (and similarly $L(j\lambda)$) can be computed from the Weyl Dimension Formula. This can be found, in particular, on p.336 of \cite{GoW}.

\begin{WDF}
Let $\lambda\in P_+(G)$. Then
\begin{center}
$\dim(L(\lambda))=\displaystyle\prod_{\alpha\in\Phi^+}\frac{(\lambda+\rho, \alpha)}{(\rho,\alpha)}$,
\end{center}
where $(\cdot,\cdot)$ is the non-degenerate bilinear form on $\mathfrak{h}^*$ induced by the killing form, and $\rho$ is half the sum of the positive roots.
\end{WDF}

Note that, as $(\lambda +\rho,\alpha)=(\lambda,\alpha)+(\rho,\alpha)$, we can re-write the terms in the product as
\begin{center}
$1+c_\lambda(\alpha)$,
\end{center}
where we use the notation from \cite{GW} that
\begin{center}
$c_\lambda(\alpha)=\displaystyle\frac{(\lambda,\alpha)}{(\rho,\alpha)}$.
\end{center}

Once again, let $\lambda\in P_+(G)$, and let $L^*(\lambda)$ denote the dual representation to $L(\lambda)$. Let $f\in L^*(\lambda)$ denote a generator of the unique line in $L^*(\lambda)$ that is fixed by $B$. The stabilizer of $f$ is a parabolic subgroup, which we denote by $P_\lambda$. Equivalently, $P_\lambda$ is the stabilizer of the hyperplane, $H\subset L(\lambda)$, that is annihilated by $f$. Note that all parabolic subgroups of $G$ arise in this way. Denote by $\mathbb{P}^*(L(\lambda))$ the projective space of \emph{hyperplanes} in $L(\lambda)$. Finally, let
\begin{center}
$\pi_\lambda:G/P\rightarrow \mathbb{P}^*(L(\lambda))$
\end{center}
be defined by mapping the coset $gP$ to the hyperplane $g(H)$. $\pi_\lambda$ is injective, and its image is the unique closed orbit of $G$ on $\mathbb{P}^*(L(\lambda))$. See \cite{FH} for details.

In \cite{GW}, the authors compute the Hilbert polynomial and Hilbert series of this embedding. We write $X_\lambda$ (or just $X$ if $\lambda$ is clear from context) to denote the image of $G/P_\lambda$ under $\pi_\lambda$. In \S2 of \cite{GW}, the authors show that the Hilbert polynomial of $X$ is given by the Weyl Dimension Formula.

\begin{thm}[Gross and Wallach]
The Hilbert polynomial of $X_\lambda$ is given by
\begin{center}
$HP_\lambda(x)=\displaystyle\prod_{\alpha\in\Phi^+}(1+x\cdot c_\lambda(\alpha))$,
\end{center}
and the Hilbert series of $X_\lambda$ is given by
\begin{center}
$HS_\lambda(x)=\displaystyle\prod_{\alpha\in\Phi^+}\left(1+c_\lambda(\alpha)\frac{d}{dx}\right)\frac{1}{1-x}$.
\end{center}
\end{thm}

A consequence of the above theorem is that the Hilbert polynomial and the Hilbert function of $X$ agree for all $n\geq0$. Varieties where this property holds are often called \emph{Hilbertian varieties}. See \cite{AK}, \cite{Ku}, \cite{SY} for in-depth treatments of such varieties in other contexts. Another direct consequence of the above theorem, not mentioned in \cite{GW}, is the following.

\begin{cor}
The coefficients of $HP_\lambda(x)$ form a strictly log-concave (and hence unimodal) sequence.
\end{cor}
\begin{proof}
By the description of $HP_\lambda(x)$ given above, the roots of the polynomial are all real and negative. Thus, by Theorem 4.5.2 in \cite{Wi}, the coefficients form a strictly log-concave sequence.
\end{proof}

\subsection{Prelimaries from rational generating functions}

The authors of \cite{GW} use representation theory to compute the formula for the Hilbert series of $G/P$. In particular, they show that this Hilbert series converges to a finite product of ordinary differential operators acting on the geometric series. In the present paper, we instead use results from the theory of rational generating functions to compute the terms in the numerator of the Hilbert series directly.

Let $HP_\lambda(x)$ denote the Hilbert polynomial of $X_\lambda$. As $X_\lambda$ is Hilbertian, the Hilbert series is given by
\begin{center}
$HS_\lambda(x)=\displaystyle\sum_{n\geq0}HP_\lambda(n)x^n$.
\end{center}
The series then converges to a rational function
\begin{center}
$HS_\lambda(x)=\displaystyle\frac{g(x)}{(1-x)^{d+1}}$,
\end{center}
where $d=\dim(G/P)$ and $g(x)$ has integer coefficients (see \cite{AM}). The coefficients of $HS_\lambda(x)$ are given by the polynomial $HP_\lambda(x)$, and so more can be said about them. In this case, we can recover the coefficients of $g(x)$ using the values of $HP_\lambda(x)$. In particular, we can use the following.

\begin{lem}[See Ch.4 of \cite{RS} \S3 of \cite{BDL}]
Let $p\in\mathbb{R}[n]$ be a polynomial of degree $d$ with generating function
\begin{center}
$S_p(x)=\displaystyle\frac{a_0+a_1x+\dots+a_dx^d}{(1-x)^{d+1}}$.
\end{center}
Then, the polynomial $p$ can be recovered from the $a_i$ as
\begin{equation}
p(n)=\sum_{j=0}^d a_j{d+n-j\choose d}.
\end{equation}
\end{lem}

This gives us a method of computing the coefficients, $a_i$, in a recursive manner. As the values of $HP_\lambda(n)=\dim(L(n\lambda))$ for all $n$, this amounts to writing the coefficients of $g(x)$ as a simple combination of $\{\dim(L(j\lambda)|0\leq j\leq d\}$. To that end, let $D_i:=\dim(L(i\lambda))$. Evaluate (1.1) at $i$ to get
\begin{equation}
D_i=\displaystyle\sum_{j=0}^d a_j{d+i-j\choose d}.
\end{equation}
For $j>i$, the terms of (1.1) are zero. We can then solve (1.1) for $a_i$ to get
\begin{equation}
a_i=D_i-{d+1\choose d}a_{i-1}-{d+2\choose d}a_{i-2}-\dots-{d+i\choose d}a_0.
\end{equation}
We will use (1.3) in the proof to compute every $a_i$ in terms of $\dim(L(j\lambda))$ for $j=0,1,\dots,d$.

\section{Proof of main theorem}

\begin{thm}
Let $d=\dim(G/P_\lambda)$. Assume $g(x)=a_0+a_1x+\dots+a_dx^d$ is the numerator polynomial in
\begin{center}
$HS_\lambda(x)=\displaystyle\frac{g(x)}{(1-x)^{d+1}}$.
\end{center}
Then $a_0=1$, and for $1\leq i\leq d$, we have
\begin{center}
$a_i=\displaystyle\sum_{j=0}^i(-1)^j{d+1\choose j}D_{i-j}$,
\end{center}
where $D_j:=\dim(L(j\lambda))$ for $j=0,\dots,d$.
\end{thm}

\begin{proof}
The proof is by induction on $i$. Before beginning, note that $D_0=1$, so $a_0=1$. Also, note that, by (1.3),
\begin{equation}
a_i=D_i-{d+1\choose d}a_{i-1}-{d+2\choose d}a_{i-2}-\dots-{d+i\choose d}a_0.
\end{equation}

Assume $i=1$. Then by (2.1), we have $a_0=1$ and $a_1=D_1-(d+1)$, consistent with the claim. Assume for induction that the claim holds for all $i=0,1,\dots,k-1$. Let $i=k$. By (2.1), we have
\begin{equation}
a_k=D_k-{d+1\choose d}a_{k-1}-{d+2\choose d}a_{k-2}-\dots-{d+k\choose d}a_0.
\end{equation}
By the inductive assumption, we can rewrite (2.2) as follows.
\begin{center}
$a_k=D_k-\displaystyle{d+1\choose d}\sum_{j=0}^{k-1}(-1)^j{d+1 \choose j}D_{k-1-j}$\\ \hspace{3cm}$-\displaystyle{d+2\choose d}\sum_{j=0}^{k-2}(-1)^j{d+1\choose j}D_{k-2-j}-\dots-{d+k\choose d}$
\end{center}
In the above, $D_{k-1}$ only appears in the first sum, $D_{k-2}$ in the first two sums, and so on. The coefficient of $D_{k-1}$ is $-{d+1\choose d}=-{d+1\choose 1}$, and the coefficent of $D_{k-2}$ is
\begin{equation}
\displaystyle{d+1\choose d}{d+1\choose 1}-{d+2\choose d}={d+1\choose 2},
\end{equation}
as desired.

In general, $D_{k-l}$ appears in the leftmost $l$ sums in $a_k$. This coefficient is
\begin{center}
$\displaystyle-{d+1\choose d}(-1)^{l-1}{d+1\choose l-1}-{d+2\choose d}(-1)^{l-2}{d+1\choose l-2}-$\\ $\dots-\displaystyle{d+l\choose d}(-1)^{l-l}{d+1\choose l-l}$,
\end{center}
or, more compactly,
\begin{equation}
-\sum_{j=1}^l(-1)^{l-j}{d+j\choose d}{d+1\choose l-j}.
\end{equation}

We claim that (2.4) is equal to $\displaystyle(-1)^l{d+1\choose l}$. Indeed, computing the sum in Mathematica gives the following
\begin{equation}
\frac{(-1)^{l+1}(l-d-2)}{l}{d+1 \choose l-1}=\frac{(-1)^l(d-l+2)}{l}{d+1\choose l-1}.
\end{equation}
Expanding the binomial and simplifying yields
\begin{center}
$\displaystyle\frac{(-1)^l(d-l+2)(d+1)!}{l(l-1)!(d-l+2)!}=(-1)^l{d+1\choose l}$,
\end{center}
as desired.

\end{proof}

\begin{cor}
Under the assumptions of the main theorem, we have
\begin{center}
$a_1=\dim(L(\lambda))-(\dim(G/P_\lambda)+1)$.
\end{center}
\end{cor}

The corollary gives algebraic meaning to $a_1$. It is, in essence, the difference between the dimension of the representation defining $P_\lambda$ and the partial flag variety $G/P_\lambda$. We will return to this in \S4 for the fundamental representations of $SL(n+1,\mathbb{C})$.

\section{The highest root of the special linear group}

As a first example to show how quickly the formula in the main theorem can be used, consider $G=SL(3,\mathbb{C})$ and $\lambda=\omega_1$, where $\omega_1$ is the first fundamental dominant weight of $G$. The positive roots of $G$ are $\Phi^+=\{\alpha_1,\alpha_2,\alpha_1+\alpha_2\}$, and so $\dim(G/P_{\omega_1})=1$ (as there are two positive roots with $c_{\omega_1}(\alpha)\neq0$). We also need $\dim(L(\omega_1))=3$. We have
\begin{center}
$a_0=1$\\
$a_1=3-(2+1)=0$.
\end{center}
And thus, $HS_{\omega_1}(x)=1/(1-x)^2$. As the dimensions of $L(\lambda)$, $L(2\lambda)$, etc. are easily computed from the Weyl Dimension Formula, the Hilbert series is very simple to compute. The same computation could be used to show that $HS_{\omega_2}(x)=1/(1-x)^2$ as well.

Now, let $\lambda$ be the highest root of $SL(3,\mathbb{C})$. In particular $\lambda=\omega_1+\omega_2$. Then, $\dim(G/P)=2$, and so we need $\dim(L(\lambda))=8$ and $\dim(L(2\lambda))=27$. Then it is easy to check that $a_0=1$, $a_1=4$, and $a_2=1$. Thus,
\begin{center}
$HS_\lambda(x)=\displaystyle\frac{1+4x+x^2}{(1-x)^3}$
\end{center}

We now generalize this example. Let $G=SL(n+1,\mathbb{C})$, and $\lambda=\omega_1+\omega_n$ be the highest root for $n>2$. Note that in this case, $L(\lambda)$ is the adjoint representation of $G$, and so we refer to $G/P_\lambda$ as the \textit{adjoint variety} of $G$. The number of positive roots, $\alpha$, such that $c_\lambda(\alpha)\neq0$ is $2n-1$, as two new such roots are added when we increase the rank of $G$ (and, therefore, $\dim(G/P_\lambda)=2n-2$). Thus, in order to use the main theorem, we need to compute
\begin{center}
$\dim(L(\lambda)), \dim(L(2\lambda)), \dim(L(3\lambda)),\dots,\dim(L((2n-2)\lambda))$.
\end{center}
This may seem like a tall order, but for low rank groups, these are not hard to compute. The following table contains the numerators of the Hilbert series for $n=3,4,5$, computed using the theorem in \S2.
\vspace{0.5cm}
\begin{center}
$\begin{array}{l|c}
 & g(x)\\
 & \\ 
\hline
 & \\ 
n=3 & 1+10x+19x^2+20x^3+20x^4\\
 & \\ 
n=4 & 1+17x+53x^2+69x^3+70x^4+70x^5+70x^6\\
 & \\ 
n=5 & 1+26x+126x^2+226x^3+251x^4+252x^5+252x^6+252x^7+252x^8
\end{array}$
\end{center}
\vspace{0.5cm}
The process for computing these is simple. The necessary dimensions are easily computed using computer software (we used LiE, see \cite{LiE}). Once these dimensions are computed, it is trivial to compute the Hilbert series using the main theorem.

In this particular instance, we can say more. Returning to the Weyl Dimension Formula, note that
\begin{center}
$\dim(L(k\lambda))=\displaystyle\prod_{\alpha>0}(1+kc_\lambda(\alpha))$.
\end{center}
For the highest root of $SL(n+1,\mathbb{C})$, there are $2n-1$ roots where $c_\lambda(\alpha)\neq0$. These are the roots $\alpha_i+\alpha_{i+1}+\dots+\alpha_{i+j}$ that either begin at $\alpha_1$ or end at $\alpha_n$. The values of $c_\lambda(\alpha)$ at these roots are 
\begin{center}
$1,1,\displaystyle\frac{1}{2},\frac{1}{2},\dots,\frac{1}{n-1},\frac{1}{n-1},\frac{2}{n}$. 
\end{center}
Therefore, we have
\begin{center}
$\dim(L(k\lambda))=\displaystyle\left(\prod_{i=1}^{n-1}\left(1+\frac{k}{i}\right)\right)^2\left(1+\frac{2k}{n}\right)$.
\end{center}
We can write this more compactly as
\begin{center}
$\dim(L(k\lambda))=\displaystyle\left(\displaystyle\frac{(k+1)^{(n-1)}}{(n-1)!}\right)^2\left(1+\frac{2k}{n}\right)$,
\end{center}
where $x^{(n)}$ denotes the rising factorial. Combining the terms yields
\begin{center}
$\displaystyle\frac{((k+1)^{(n-1)})^2(n+2k)}{(n-1)!n!}$.
\end{center}
Converting to falling factorials, we have
\begin{center}
$\dim(L(k\lambda))=\displaystyle\frac{((n+k-1)_{n-1})^2(n+2k)}{(n-1)!n!}$.
\end{center}
We denote this by $\dim(L(k\lambda))=P(n,n+k)$. See \cite{OEIS} for the case when $n=4$, and a brief general description of these numbers. In particular, $P(n,n+k)$ counts the number of of $(n+k)$-step paths that start at $(0,0)$ and reach $(n,n)$ for the first time, where the only moves allowed are to increase the first coordinate by one, to increase the second coordinate by one, to increase both coordinates by one, or to leave the point unchanged.

With this, we can now compute $\dim(L(k\lambda))$ very quickly for any $k$. In particular, from the main theorem, we have
\begin{center}
$a_i=\displaystyle\sum_{j=0}^i(-1)^j{d+1\choose j}D_{i-j}$.
\end{center}
Computing $D_{i-j}$ yields
\begin{center}
$a_i=\displaystyle\sum_{j=0}^i(-1)^j{d+1\choose j}\frac{(n+i-j-1)_{n-1})^2(n+2i-2j)}{(n-1)!n!}$\\
$=\displaystyle\sum_{j=0}^i(-1)^j{d+1\choose j}P(n,n+i-j)$.
\end{center}

\section{The fundamental representations of the special linear group}

In this concluding section, we compute the term $a_1$ in the Hilbert series of $G/P_\lambda$, where $\lambda$ is a fundamental dominant weight of $SL(n+1,\mathbb{C})$. This term has both algebraic significance, as it measure the difference between the dimension of $L(\lambda)$ and the dimension of $G/P_\lambda$, and combinatorial significance, as $a_1$ is the difference between certain binomial coefficients and the members of the rascal triangle (see \cite{Ra}). For all that follows, assume that $G=SL(n+1,\mathbb{C})$, $n=\text{rank(G)}$, $\Delta=\{\alpha_1,\dots,\alpha_n\}$, and $\omega_1,\dots,\omega_n$ are the fundamental dominant weights.

We first compute the number of terms, $c_\lambda(\alpha)$ which are nonzero, where $\alpha\in\Phi^+$, and $\lambda$ is a fundamental dominant weight. Recall that this is the dimension of $G/P_\lambda$.

\begin{prop}
$\dim(G/P_{\omega_i})=(n-i+1)i$, where $i=1,2,\dots,n$.
\end{prop}

\begin{proof}
The positive roots of $G$ have the following form
\begin{center}
$\Phi^+=\{\sum_{j=k}^n\alpha_j\mid 1\leq k\leq n\leq l\}$.
\end{center}
The positive roots which yield nonzero $c_{\omega_i}(\alpha)$ are precisely those roots which have $\alpha_i$ in the sum. We count these systematically.

First, the number of roots whose initial term is $\alpha_i$ is $n-i+1$. These are precisely the roots
\begin{center}
$\alpha_i, \alpha_i+\alpha_{i+1},\dots, \alpha_i+\dots+\alpha_n$.
\end{center}
Next, we count the number of roots whose sum starts at $\alpha_{i-1}$, which include $\alpha_i$. There are again $n-i+1$ such roots. This can be seen by keeping track of the end term in each root sum. These are the roots
\begin{center}
$\alpha_{i-1}+\alpha_i,\alpha_{i-1}+\alpha_i+\alpha_{i+1},\alpha_{i-1}+\dots +\alpha_n$.
\end{center}
Continuing this, there are $n-i+1$ roots starting at $\alpha_{i-2}$ that include $\alpha_i$, and so forth. This yields exactly $i$ sets of roots that contain $n-i+1$ roots each, finishing the proof.
\end{proof}

The sequence of numbers $(n-i+1)i$ have combinatorial significance. For more information see \cite{OEIS3}.

Note that $\dim(L(\omega_i))=\displaystyle{n+1\choose i}$. Therefore, we have
\begin{center}
$a_1=\displaystyle{n+1\choose i}-(i(n-i+1)+1)$.
\end{center}
For $\omega_1$ and $\omega_n$, it is easy to verify that $a_1=0$. For $i=2,\dots,l-1$, $h_1=T(n+1,i)$, where $T(n,k)$ is the sequence at \cite{OEIS4}. This is the difference of the entries in Pascal's triangle and the entries in the rascal triangle, found in \cite{Ra}. The rascal triangle is formed similarly to Pascal's triangle. Start with two rows: 1 and 1,1. Form the next row by taking (East$\times$West+1)/North. Thus, the next row would be 1, 2, 1, and so forth. To get the numbers in the $n^{\text{th}}$ diagonal of the rascal triangle, you start with one, and add $n$ over and over again. The $n^{\text{th}}$ diagonal of Pascal's triangle is formed by starting with 1, adding $n$, and then adding numbers greater than or equal to $n$. Putting this together, we have
\begin{prop}
Let $G=SL(n+1,\mathbb{C})$ and $\lambda=\omega_i$. Then $a_1$ is non-negative.
\end{prop}


\end{document}